\documentclass[10pt]{amsart}

\usepackage{mathtools}
\usepackage{geometry}
\usepackage{amsmath}
\usepackage{graphicx}
\usepackage[latin1]{inputenc}
\usepackage{comment}
\usepackage{qtree,array,youngtab}

\usepackage{color, graphicx}
\usepackage{soul}
\usepackage[T1]{fontenc}

\newcommand{\la}{\lambda}

\definecolor{darkgreen}{rgb}{0.,0.5,0.}

\numberwithin{equation}{section}
\newtheorem{theorem}{Theorem}[section]
\newtheorem{lemma}[theorem]{Lemma}

\newtheorem{cor}[theorem]{Corollary}

\newtheorem{ex}[theorem]{Example}

\allowdisplaybreaks

\title{On the largest sizes of certain simultaneous core partitions with distinct parts}

\author{Huan Xiong}
\address{Universit\'e de Strasbourg, CNRS, IRMA UMR 7501, F-67000 Strasbourg, France} \email{xiong@math.unistra.fr}



\begin{document}
\begin{abstract} Motivated by Amdeberhan's conjecture on $(t,t+1)$-core partitions with distinct parts,  various results on the numbers, the largest sizes and the average sizes of simultaneous core partitions with distinct parts were obtained by many mathematicians recently. In this paper, we derive the largest sizes of $(t,mt\pm 1)$-core partitions with distinct parts, which verifies a generalization of  Amdeberhan's conjecture. We also prove that the numbers of such partitions with the largest sizes are at most $2$.
\end{abstract}

 \maketitle

\section{Introduction} Recall that a \emph{partition} is a finite weakly decreasing sequence of positive integers
$\lambda = (\lambda_1, \lambda_2, \ldots, \lambda_\ell)$ where 
 $\la_i\ (1\leq i\leq \ell)$ are called the \emph{parts} and $\sum_{1\leq i\leq \ell}\lambda_i$  the
\emph{size} of $\lambda$ (see \cite{Macdonald,ec2}). We  associate a partition $\lambda$ with its \emph{Young
diagram}, which is an array of boxes arranged in left-justified
rows with $\lambda_i$ boxes in the $i$-th row. For the $(i,j)$-box in
the $i$-th row and $j$-th column in the Young diagram, its \emph{hook length} $h(i, j)$ is defined to be
the number of boxes directly to the right, and directly below, including the
box itself. Let $t$ be a positive integer.  A partition $\lambda$ is
called a \emph{$t$-core partition} if none of its hook lengths is a
multiple of $t$. Furthermore,  $\lambda$ is called a
\emph{$(t_1,t_2,\ldots, t_m)$}-core partition if it is
simultaneously  a $t_1$-core, a $t_2$-core, $\ldots$, a $t_m$-core
partition (see \cite{tamd, KN}). For example, the Young diagram and hook
lengths of the partition $(7,2,1)$ are given in Figure \ref{fig:1}. Therefore it is a $(6, 8)$-core partition since none of its
hook lengths is divisible by $6$ or $8$.

\begin{figure}[htbp]
\begin{center}
\Yvcentermath1

\begin{tabular}{c}
$\young(9754321,31,1)$

\end{tabular}

\end{center}
\caption{The Young diagram and hook lengths of the partition
$(7,2,1)$.} \label{fig:1}
\end{figure}

Simultaneous core partitions have been widely studied in the past fifteen years (see \cite{tamd1,   AHJ, CHW, ford, PJ,NS, N3, ols, SZ,Wang,Xiong1,  YZZ}) since Anderson's work \cite{and}, who showed that the number of $(t_1,t_2)$-core
partitions is equal to $  (t_1+t_2-1)!/(t_1!\, t_2!)$, where $t_1$
and $t_2$ are coprime to each other.  Olsson and Stanton
\cite{ols} proved that the largest size of $(t_1,t_2)$-core
partitions is $    {(t_1^2-1)(t_2^2-1)}/{24}$. The average size of such partitions is conjectured to be $  {(t_1-1)(t_2-1)(t_1+t_2+1)}/{24}$ by Armstrong \cite{AHJ} and first proved by Johnson \cite{PJ}.

The problem of evaluating the number of simultaneous core partitions with distinct parts was raised by Amdeberhan \cite{tamd}. He conjectured that the number of $(t,t+1)$-core partitions with distinct parts is equal to the $(t+1)$-th Fibonacci number.  Amdeberhan  also made several conjectures concerning the  largest size and the average size of such partitions. Amdeberhan's conjectures were  proved independently by Straub \cite{Straub}, Nath-Sellers \cite{NS}, Zaleski \cite{Za}, and the author \cite{Xiong} recently. Zaleski \cite{Za} also computed the explicit formulas for the moments of the sizes of these partitions. Straub \cite{Straub} gave several conjectures on
the number and the largest size of $(t,t+2)$-core partitions with distinct parts, which were first proved by Yan, Qin, Jin and Zhou \cite{YQJZ}. Later, Baek, Nam and Yu \cite{BNY} gave a bijective proof for the number of such partitions.  The explicit formulas for the
 moments of
the sizes of these partitions were obtained by Zaleski and Zeilberger \cite{ZZ}.

Recently,  Straub \cite{Straub} and Nath-Sellers \cite{NS2}  derived closed formulas for the numbers of $(t, mt-1)$ and $(t, mt+1)$-core partitions with distinct parts respectively. Therefore it is natural to try to find the largest sizes and average sizes of these two kinds of simultaneous core partitions. For the average size of $(t, mt-1)$-core partitions with distinct parts, Zaleski conjectured  explicit formulas in \cite{Za2}.

In this paper, we obtain the largest sizes of $(t, mt+1)$ and $(t, mt-1)$-core partitions with distinct parts, and determine the numbers of such partitions with the largest sizes. For simplicity, let 
\begin{align}
\displaystyle \alpha_{m,t}(x):=\frac{mt+t+x}{2(m+2)}
\end{align}
for any $m,t\in \mathbb{N}$ and $x\in \mathbb{Z}$.
We will prove the following result in Section \ref{sec:mt+1}.

\begin{theorem} \label{th:main}
Let $t\geq 2$ and $m\geq 1$ be two given positive integers.
The largest size of $(t, mt + 1)$-core partitions with distinct parts is 
\begin{align*}
\displaystyle 
\begin{cases} -\lfloor   \alpha_{m,t}(1)\rfloor ^2\cdot  (m^2+2m)/{2}\, + \, \lfloor   \alpha_{m,t}(1)\rfloor \cdot  (m^2t+mt+m)/{2}, \quad &\ \text{if}\  \{  \alpha_{m,t}(1)\}\leq  1/2; \\ -\lfloor   \alpha_{m,t}(1)+1\rfloor ^2\cdot  (m^2+2m)/{2}\, + \, \lfloor   \alpha_{m,t}(1)+1\rfloor \cdot  (m^2t+mt+m)/{2}, \quad &\ \text{if}\  \{  \alpha_{m,t}(1)\}>  1/2,
\end{cases}
\end{align*}
where  $\lfloor x\rfloor $ denotes the largest integer not greater than the real number $x$, and $\{x\}=x-\lfloor x\rfloor $ denotes the fractional part of $x$.
The number of such partitions with the largest size is  
\begin{align*}
\displaystyle 
\begin{cases} 
2, \quad & \text{if}\  \{  \alpha_{m,t}(1)\}=  1/2, \\ 1, \quad & \text{if}\  \{  \alpha_{m,t}(1)\}\neq   1/2.
\end{cases}
\end{align*}
\end{theorem}

\begin{ex}
$(1)$ Let $t=7$, $m=3$. Then  $\{  \alpha_{m,t}(1)\}=\{  {29}/{10}\}=  {9}/{10}>  1/2$. For the $(7, 22)$-core partitions with distinct parts,   the largest size is $63$, and the number of such partitions with the largest size is $1$ by Theorem \ref{th:main}. Actually we can check that $(12,11,10,8,7,6,4,3,2)$ is the only such partition with the largest size $63$.  

$(2)$ Let $t=10$, $m=1$. Then  $\{  \alpha_{m,t}(1)\}=\{  {7}/{2}\}=  1/2$. For the $(10, 11)$-core partitions with distinct parts,  the largest size is $18$, and the number of such partitions with the largest size is $2$ by Theorem~\ref{th:main}. In fact, $(7,6,5)$ and $(6,5,4,3)$ are the only two such partitions with the largest size $18$.
\end{ex}

It is surprising that  the case for $(t, mt-1)$-core partitions with distinct parts is much more complicated than $(t, mt+1)$-core partition case.   We obtain the following result in Section \ref{sec:mt-1}. Here we assume $m\geq 2$ in Theorem  \ref{th:main*} since $m=1$ case can be implied by Theorem \ref{th:main}.
\begin{theorem} \label{th:main*}
Let $t\geq 2$ and $m\geq 2$ be two given positive integers. Define two functions $F$ and $G$ as follows:
\begin{align*}
\displaystyle 
F(x)=-x^2\cdot  (m^2+2m)/{2}+x\cdot  (m^2t+mt+3m)/{2}-mt 
\end{align*}
and
\begin{align*}
\displaystyle 
G(x)=-x^2\cdot  (m^2+2m)/{2}+x\cdot  (m^2t+mt-m)/{2}.
\end{align*}

$(1)$ When (i) $t \equiv 1 \pmod 2$ and $t> m+1$; or (ii) $t \equiv 0 \pmod 2$ and $t> 2m+3$, the largest size of $(t, mt - 1)$-core partitions with distinct parts is 
\begin{align*}
\displaystyle 
\begin{cases} \max \left(  F(\lfloor   \alpha_{m,t}(3)\rfloor ), G(\lfloor   \alpha_{m,t}(-1)\rfloor ) \right), \quad & \text{if}\  \{  \alpha_{m,t}(3)\}\leq  1/2 \ \text{and}\  \{  \alpha_{m,t}(-1)\}\leq  1/2; \\ \max \left(  F(\lfloor   \alpha_{m,t}(3)\rfloor ), G(\lfloor   \alpha_{m,t}(-1)\rfloor +1) \right), \quad & \text{if}\  \{  \alpha_{m,t}(3)\}\leq  1/2 \ \text{and}\  \{  \alpha_{m,t}(-1)\}>  1/2; \\
\max \left(  F(\lfloor   \alpha_{m,t}(3)\rfloor +1), G(\lfloor   \alpha_{m,t}(-1)\rfloor ) \right), \quad & \text{if}\  \{  \alpha_{m,t}(3)\}>  1/2 \ \text{and}\  \{  \alpha_{m,t}(-1)\}\leq  1/2;
\\
\max \left(  F(\lfloor   \alpha_{m,t}(3)\rfloor +1), G(\lfloor   \alpha_{m,t}(-1)\rfloor +1) \right), \quad & \text{if}\  \{  \alpha_{m,t}(3)\}>  1/2 \ \text{and}\  \{  \alpha_{m,t}(-1)\}>  1/2,
\end{cases}
\end{align*}
where $\max(x,\ y)$ denotes the maximal
number in $x$ and $y$.

$(2)$ When (iii) $t \equiv 1 \pmod 2$ and $t\leq m+1$; or (iv) $t \equiv 0 \pmod 2$ and $t\leq 2m+3$, the largest size of $(t, mt - 1)$-core partitions with distinct parts is 
\begin{align*}
\begin{cases} \max \left(  F(\lfloor   \alpha_{m,t}(3)\rfloor ), G(\lfloor \displaystyle   \frac{t-1}{2}\rfloor ) \right), \quad &\ \text{if}\  \{  \alpha_{m,t}(3)\}\leq  1/2; \\ \max \left(  F(\lfloor   \alpha_{m,t}(3)\rfloor +1), G(\lfloor \displaystyle   \frac{t-1}{2}\rfloor ) \right), \quad &\ \text{if}\  \{  \alpha_{m,t}(3)\}>  1/2.
\end{cases}
\end{align*}

$(3)$ The number of $(t, mt - 1)$-core partitions with distinct parts which have the largest size is  at most $2$.
\end{theorem}

\begin{ex}
$(1)$ Let $t=8$, $m=3$. Then $t \equiv 0 \pmod 2$ and $t\leq 2m+3$. Also we have $\{  \alpha_{m,t}(3)\}=\{  {35}/{10}\}=  1/2$. For the $(8, 23)$-core partitions with distinct parts,   the largest size is $\max( F(3), G(3) )=72$ by Theorem \ref{th:main*}. In fact, we can check that $(14,13,12,9,8,7,4,3,2)$ is the only such partition with the largest size $72$.  

$(2)$ Let $t=7$, $m=2$. Then $t \equiv 1 \pmod 2$ and $t> m+1$. Also we have $\{  \alpha_{m,t}(3)\}=\{  {24}/{8}\}=0<  1/2$ and $\{  \alpha_{m,t}(-1)\}=\{  {20}/{8}\}=  1/2$. For the $(7, 13)$-core partitions with distinct parts,   the largest size is $\max( F(3), G(2) )=24$ by Theorem \ref{th:main*}. Actually  $(7,6,5,3,2,1)$ and $(9,8,4,3)$ are the only two such partitions with the largest size $24$.
\end{ex}

The case $m=1$ in Theorem \ref{th:main} implies the following Amdeberhan's conjecture on the largest size of $(t, t + 1)$-core partitions with distinct parts.

\begin{cor}[see Conjecture 11.9 of  \cite{tamd}]\label{main} 
Let $t\geq 2$ be a positive integer.  Then the largest size of $(t, t + 1)$-core partitions with distinct parts is 
$\lfloor   t(t+1)/6\rfloor $.
The number of such partitions with the largest size is $2$ if $t  \equiv 1 \ (\text{mod}\ 3)$ and $1$ otherwise.
\end{cor}

\section{The $\beta$-sets of core partitions}
The \emph{$\beta$-set} of a partition $\lambda = (\lambda_1, \lambda_2, \ldots, \lambda_\ell)$ is denoted by
$$\beta(\lambda)=\{\lambda_i+\ell-i : 1 \leq i \leq \ell\}.$$
In fact, $\beta(\lambda)$
 is equal to the set of
hook lengths of boxes in the first column of the corresponding Young
diagram of $\la$ (see \cite{ols,Xiong1}). It is easy to see that a partition~$\lambda$ is uniquely
determined by its $\beta$-set $\beta(\lambda)$.  The following results are well-known.

\begin{lemma}[ \cite{and, berge,ols,Xiong1,Xiong}]\label{betaset}

{\em (1)}  The size of a partition
$\lambda$ is  determined by its $\beta$-set as following:
\begin{align}\label{eq:2.1}
| \lambda |=\sum_{x\in
\beta(\lambda)}{x}-\binom{ |\beta(\lambda)| }{2}.
\end{align}

{\em (2)} {\em (The abacus condition for t-core partitions.)}
A partition $\lambda$ is a $t$-core partition if and only if for any
$x\in \beta(\lambda)$ with $x\geq t$,
we have $x-t \in \beta(\lambda)$.
\end{lemma}

The following result is obvious by the definition of $\beta$-sets.
\begin{lemma} \label{distinct}
The partition  $\lambda$ is a partition with distinct parts if and only if there does not exist $x,y\in \beta(\lambda)$ with $x-y=1.$
\end{lemma}

An integer $x\in\beta(\la)$ is called \emph{$t$-maximal} in $\beta(\la)$ if $x+t\notin\beta(\la)$. For each $t$-core partition $\la$ and $1\leq i \leq t-1$, let  $n_i(\la)$ be the number of integers $x$ satisfying $x\in \beta(\la)$ and $x  \equiv i \ (\text{mod}\ t)$. Then by Lemma \ref{betaset}(2) we know $t(n_i(\la)-1)+i$ is $t$-maximal in $\beta(\la)$ if  $i\in \beta(\la)$, and $n_i(\la)=0$ if  $i\notin \beta(\la)$. Furthermore, 
$$
\beta(\la)=\bigcup_{i=1}^{t-1}\, \bigcup_{j=0}^{n_i(\la)-1}\{jt+i\}.
$$
Therefore $|\beta(\la)|=\sum_{i=1}^{t-1}n_i(\la)$ and by \eqref{eq:2.1} we obtain
\begin{equation}\label{eq:size_formula}
|\la|=\sum_{i=1}^{t-1}\left(in_i(\la)+t\binom{n_i(\la)}{2}\right)-\binom{\sum_{i=1}^{t-1}n_i(\la)}{2}.
\end{equation}

\begin{ex}
Let $\lambda=(7,2,2)$. Then $\beta(\lambda)=\{9,3,2\}$.  By Lemma \ref{betaset}(2) we know that $\la$ is a $(6, 7)$-core partition. Also, $\la$ is \textit{not} a partition with distinct parts. We can check that  $3,2 \in \beta(\lambda)$ and $3-2=1$.
Let $t=6$. Then $\displaystyle 9,2 \in \beta(\lambda)$ are $t$-maximal in $\beta(\la)$ while $3 \in \beta(\la)$ is not. We have
$n_2(\la)=1$, $n_3(\la)=2$,  and $n_i(\la)=0$ otherwise. In this case \eqref{eq:size_formula} holds since $$7+2+2=2 \cdot 1+6\cdot \binom{1}{2}+3\cdot 2+6\cdot \binom{2}{2}- \binom{1+2}{2}.$$
\end{ex}

\section{The largest size of $(t,mt+1)$-core partitions with distinct parts} \label{sec:mt+1}

Let $S^+_{m,t}$ be the set of $(t, mt + 1)$-core partitions with distinct parts, and $
\mathcal{C}^+_{m,t}$  be the set of all sequences $ (x_1,x_2,\ldots,x_{t-1})\in\mathbb{N}^{t-1}$ with $0\leq x_i \leq m\ (1 \leq i \leq t-1)$ and $x_ix_{i+1}=0 \ ( 1 \leq i \leq t-2)$.
For each $\la \in S^+_{m,t}$, let $\psi(\la):=(n_1(\la),n_2(\la),\ldots,n_{t-1}(\la))$. 
\begin{theorem}\label{th:bijection}
The map $\psi$ provides a bijection between the sets $ S^+_{m,t}$ and 
$\mathcal{C}^+_{m,t}$.
\end{theorem}
\begin{proof} 
Let $\la \in S^+_{m,t}$. By Lemma \ref{betaset}(2)  we have $mt,mt+1\notin \beta(\lambda)$.   If $x\in\beta(\lambda)$ with $x\geq mt+2$, then $x-mt,x-(mt+1)\in \beta(\lambda)$ by Lemma \ref{betaset}(2). But by Lemma \ref{distinct} this is impossible since $\lambda$ is a partition with distinct parts. Then, we know  that $x\notin \beta(\lambda)$ and thus $\beta(\lambda)$ must be a subset of  $\{ 1,2,\ldots,mt-1\}$. Therefore, $0\leq n_i(\la) \leq m$ for any $1\leq i\leq t-1$. Furthermore, by Lemma \ref{distinct} we have $i$ and $i+1$ can not be in $\beta(\la)$ at the same time, which means that $n_i(\la)n_{i+1}(\la)=0$ for any $1\leq i\leq t-2$. By the above arguments we have $\psi(\la)\in \mathcal{C}^+_{m,t}$ for any $\la \in S^+_{m,t}$. 

On the other hand, for any $(x_1,x_2,\ldots,x_{t-1})\in \mathcal{C}^+_{m,t}$, let $\la$ be the partition with the $\beta$-set 
$$
\beta(\la)= \bigcup\limits_{i=1}^{t-1}\{ k_it+i: 0\leq k_i\leq x_i-1 \}.
$$
Then, Lemmas \ref{betaset}(2) and \ref{distinct} imply $\la \in S^+_{m,t}$ and $\psi(\la)=(x_1,x_2,\ldots,x_{t-1})$. Therefore, $\psi$ is a bijection between the sets $S^+_{m,t}$ and 
$\mathcal{C}^+_{m,t}$.
\end{proof}

Now we are ready to give the proof of Theorem \ref{th:main}.
\begin{proof}[ Proof of Theorem \ref{th:main}]
Suppose that $\la$ is a partition in $S^+_{m,t}$ with the largest size.

\textbf{Step 1.}  We claim that $n_i(\la)=0$ or $m$ for  $1\leq i \leq t-1$. Otherwise, suppose that there exists some~$i$ such that $1\leq n_i(\la) \leq m-1$. We define
$$
\la'=\psi^{-1}(n_1(\la),n_2(\la),\ldots,n_{i-1}(\la),n_i(\la)+1,n_{i+1}(\la),\ldots,n_{t-1}(\la))
$$ 
and 
$$
\la''=\psi^{-1}(n_1(\la),n_2(\la),\ldots,n_{i-1}(\la),n_i(\la)-1,n_{i+1}(\la),\ldots,n_{t-1}(\la)).
$$
Then,  Theorem \ref{th:bijection} implies that $\la'$ and $\la''$
are also in $S^+_{m,t}$. By \eqref{eq:size_formula} or \eqref{eq:2.1}  we have
$$
| \la' |-| \la |= tn_i(\la)+i-\sum_{i=1}^{t-1}n_i(\la)
$$ 
and
$$
| \la |-| \la'' |= t(n_i(\la)-1)+i-(\sum_{i=1}^{t-1}n_i(\la)-1).
$$ 
Putting these two equalities together yields
\begin{eqnarray*}
2| \la |   =  | \la'|+|
\la''|-(t-1)
< | \la'|+| \la''|.
\end{eqnarray*}
This contradicts the assumption that $\la$ is a partition in $S^+_{m,t}$ with the largest size. Then 
the claim holds.

\textbf{Step 2.} By Step $1$ we have $n_i=0$ or $m$ for any $1\leq i \leq t-1$. Let 
$$
I=\{  1\leq i \leq t-1: \  n_i(\la)= m   \}.
$$
Since $\la$ has the largest size in $S^+_{m,t}$, we know that  $I=\{ t+1-2j:\ 1\leq j \leq r    \}$  for some integer $1\leq r \leq \lfloor   t/2\rfloor $. Let $\la^k$ be the partition in $S^+_{m,t}$ such that $n_i(\la^k)=m$ if $i=t+1-2j $ for some $1\leq j \leq k$; and $n_i(\la^k)=0$ otherwise. Then, $\la \in \{\la^r:\  1\leq r \leq \lfloor   t/2\rfloor  \}$. 

\textbf{Step 3.} By \eqref{eq:size_formula} the size of $\la^r$ is 
\begin{align*}
|\la^r|
&= \sum_{j=1}^{r}\left(m(t+1-2j)+t\binom{m}{2}\right)-\binom{mr}{2}
\\&=
-r^2\cdot  (m^2+2m)/{2}+r\cdot  (m^2t+mt+m)/{2}
\\&= 
-  \frac{m^2+2m}{2}\cdot(r-  \alpha_{m,t}(1))^2+  \frac{m^2+2m}{2}\cdot  \alpha^2_{m,t}(1)
.
\end{align*}

Case $1$: When $t \equiv 1 \pmod 2$ and $t\leq m+3$, we have 
$$
  \frac{t-1}{2} \leq   \frac{t-1}{2}+  \frac{m-t+3}{2(m+2)}\leq    \alpha_{m,t}(1)<  \frac{t-1}{2}+  \frac12
.
$$
This means that $\{  \alpha_{m,t}(1)\}<  1/2$ since $  (t-1)/{2}\in \mathbb{N}$. Notice that $1\leq r \leq   (t-1)/{2}$. Then $\la^r$ has the largest size in $S^+_{m,t}$ if and only if 
$$r=  \frac{t-1}{2}=\lfloor   \alpha_{m,t}(1)\rfloor .$$

Case $2$: When $t \equiv 0 \pmod 2$ or $t> m+3$, we have 
$$
  \frac{1}{2}<  \alpha_{m,t}(1)=  \frac{t}{2}-  \frac{t-1}{2(m+2)}=  \frac{t-1}{2}+  \frac{m-t+3}{2(m+2)}< \lfloor   \frac{t}{2}\rfloor 
.
$$
Notice that $1\leq r \leq \lfloor   t/2\rfloor $. Then $\la^r$ has the largest size in $S^+_{m,t}$ if and only if 
\begin{align*}
r= \begin{cases}  \lfloor   \alpha_{m,t}(1)\rfloor , \quad &\ \text{when}\  \{  \alpha_{m,t}(1)\}<  1/2; \\ \lfloor   \alpha_{m,t}(1)\rfloor  \ \text{or}\ \lfloor   \alpha_{m,t}(1)\rfloor +1, \quad &\   \text{when}\ \{  \alpha_{m,t}(1)\}=  1/2; \\
  \lfloor   \alpha_{m,t}(1)\rfloor +1, \quad &\  \text{when}\  \{  \alpha_{m,t}(1)\}>  1/2.
\end{cases}
\end{align*}

Finally Theorem \ref{th:main} holds by the above two cases.
\end{proof}

\section{Largest size of $(t,mt-1)$-core partitions with distinct parts} \label{sec:mt-1}

Let $S^-_{m,t}$ be the set of $(t, mt - 1)$-core partitions with distinct parts,
and $
\mathcal{C}^-_{m,t}$  be the set of all sequences  $ (x_1,x_2,\ldots,x_{t-1})\in\mathbb{N}^{t-1}$ with $0\leq x_i \leq m\ ( 1 \leq i \leq t-2)$, $x_ix_{i+1}=0 \ ( 1 \leq i \leq t-2)$ and $ 0\leq x_{t-1}\leq m-1$.

For each $\la \in S^-_{m,t}$, let $\varphi(\la):=(n_1(\la),n_2(\la),\ldots,n_{t-1}(\la))$. 
\begin{theorem}\label{th:bijection*}
The map $\varphi$ provides a bijection between the sets $ S^-_{m,t}$ and 
$\mathcal{C}^-_{m,t}$.
\end{theorem}
\begin{proof} 
Let $\la \in S^-_{m,t}$. By Lemma \ref{betaset}(2)  we have $mt,mt-1\notin \beta(\lambda)$.  If $x\in\beta(\lambda)$ with $x\geq mt+1$,  we know $x-mt,x-(mt-1)\in \beta(\lambda)$ by Lemma \ref{betaset}(2). But by Lemma \ref{distinct} this is impossible since $\lambda$ is a partition with distinct parts. Then, we know  $x\notin \beta(\lambda)$ and thus $\beta(\lambda)$ must be a subset of  $\{ 1,2,\ldots,mt-2\}$. Therefore, $0\leq n_i(\la) \leq m$ for any $1\leq i\leq t-2$ and $0\leq n_{t-1}(\la) \leq m-1$. Furthermore, by Lemma \ref{distinct} we have $i$ and $i+1$ can not be in $\beta(\la)$ at the same time, which means that $n_i(\la)n_{i+1}(\la)=0$ for any $1\leq i\leq t-2$. By the above arguments we have $\varphi(\la)\in \mathcal{C}^-_{m,t}$ for any $\la \in S^-_{m,t}$. 

On the other hand, for any $(x_1,x_2,\ldots,x_{t-1})\in \mathcal{C}^-_{m,t}$, let $\la$ be the partition with the $\beta$-set 
$$
\beta(\la)= \bigcup\limits_{i=1}^{t-1}\{ k_it+i: 0\leq k_i\leq x_i-1 \}.
$$
Then Lemmas \ref{betaset}(2) and \ref{distinct} imply $\la \in S^-_{m,t}$ and $\varphi(\la)=(x_1,x_2,\ldots,x_{t-1})$. Therefore, $\varphi$ is a bijection between the sets $S^-_{m,t}$ and 
$\mathcal{C}^-_{m,t}$.
\end{proof}

Now we are ready to give the proof of Theorem \ref{th:main*}.

\begin{proof}[ Proof of Theorem \ref{th:main*}]
Suppose that $\la$ is a partition in $S^-_{m,t}$ with the largest size.

\textbf{Step 1.}  We claim that $n_{t-1}(\la)=0$ or $m-1$, and $n_i(\la)=0$ or $m$ for $1\leq i \leq t-2$. Otherwise, assume that $1\leq n_{t-1}(\la) \leq m-2$ (in this case, let $i=t-1$) or there exists some $1\leq i \leq t-2$ such that $1\leq n_i(\la) \leq m-1$.    We define
$$
\la'=\varphi^{-1}(n_1(\la),n_2(\la),\ldots,n_{i-1}(\la),n_i(\la)+1,n_{i+1}(\la),\ldots,n_{t-1}(\la))
$$ 
and 
$$
\la''=\varphi^{-1}(n_1(\la),n_2(\la),\ldots,n_{i-1}(\la),n_i(\la)-1,n_{i+1}(\la),\ldots,n_{t-1}(\la)).
$$
Then, Theorem \ref{th:bijection*} implies that $\la'$ and $\la''$
are also in $S^-_{m,t}$. By \eqref{eq:size_formula} or \eqref{eq:2.1}  we have
\begin{eqnarray*}
2| \la |   =  | \la'|+|
\la''|-(t-1)
< | \la'|+| \la''|.
\end{eqnarray*}
This contradicts the assumption that $\la$ is a partition in $S^-_{m,t}$ with the largest size. Then 
the claim holds.

\textbf{Step 2.} By Step $1$ we obtain $n_{t-1}(\la)=0$ or $m-1$, and $n_i(\la)=0$ or $m$ for  $1\leq i \leq t-2$. Let $$I=\{  1\leq i \leq t-1: \  n_i(\la)\neq 0  \}.$$  Since $\la$ has the largest size in $S^-_{m,t}$, we know that  $I=\{ t+1-2j:\ 1\leq j \leq r    \}$ for some $1\leq r \leq \lfloor   t/2\rfloor $ or $\{ t-2j:\ 1\leq j \leq s    \}$  for some $1\leq s \leq \lfloor   (t-1)/{2}\rfloor $. Let $\la^r\ (1\leq r \leq \lfloor   t/2\rfloor )$ and $\mu^s\ (1\leq s \leq \lfloor   (t-1)/{2}\rfloor )$ be the partitions such that 
\begin{align*}
&n_{i}(\la^r)=
\begin{cases}
m, \quad &\text{if}\ i=t+1-2j  \ \text{for some}\ 2\leq j \leq r;
\\
m-1, \quad &\text{if}\ i=t-1;
\\
0, \ &\text{otherwise}.
\end{cases} 
\\&n_{i}(\mu^s)=
\begin{cases}
m, \quad &\text{if}\ i=t-2j  \ \text{for some}\ 1\leq j \leq s;
\\
0, \ &\text{otherwise}.
\end{cases} 
\end{align*}

Then $\la \in \{\la^r:\  1\leq r \leq \lfloor   t/2\rfloor  \} \cup \{\mu^s:\  1\leq s \leq \lfloor   (t-1)/{2}\rfloor  \}$. 

\textbf{Step 3.}
When $t=2$, we have $\lfloor   t/2\rfloor =1$ and $\lfloor   (t-1)/{2}\rfloor =0$, therefore by Step $2$ the partition $\la^1$ is the only partition  in $S^-_{m,t}$ with the largest size. It is easy to check that $\la^1=(m-1,m-2,m-3,\ldots,3,2,1)$ and $|\la^1|=\binom{m}{2}$.

When $t=3$,  by Step $2$ we obtain $\la \in \{\la^1, \mu^1\}$. But in this case $|\la^1|=m^2-m<m^2=|\mu^1|$. Then~$\mu^1$ is the only partition  in $S^-_{m,t}$ with the largest size $m^2$. It is easy to check that $\mu^1=(2m-1,2m-3,2m-5,\ldots,5,3,1)$. 

\textbf{Step 4.}
Next we assume that $t\geq 4$ and $m\geq 2$.
By \eqref{eq:size_formula}, the size of $\la^r$ is 
\begin{align} \label{eq: la_r}
\displaystyle 
|\la^r|
= & (m-1)(t-1)+t\binom{m-1}{2}+\sum_{j=2}^{r}\left(m(t+1-2j)+t\binom{m}{2}\right)-\binom{mr-1}{2} \nonumber
\\&=
-r^2\cdot  (m^2+2m)/{2}+r\cdot  (m^2t+mt+3m)/{2}-mt
\\&= 
-  \frac{m^2+2m}{2}\cdot(r-  \alpha_{m,t}(3))^2+  \frac{m^2+2m}{2}\cdot \alpha^2_{m,t}(3)-mt \nonumber
.
\end{align}
When $t \equiv 1 \pmod 2$ and $t\leq m+5$, we have 
$$
  \frac{t-1}{2} \leq   \frac{t-1}{2}+  \frac{m-t+5}{2(m+2)}\leq    \alpha_{m,t}(3)=  \frac{t}{2}-  \frac{t-3}{2(m+2)}<  \frac{t}{2}
.
$$
When $t \equiv 0 \pmod 2$ or $t> m+5$, we have 
$$
1<  \alpha_{m,t}(3)=  \frac{t}{2}-  \frac{t-3}{2(m+2)}=  \frac{t-1}{2}+  \frac{m-t+3}{2(m+2)}< \lfloor   \frac{t}{2}\rfloor 
.
$$
Then, in each case, $\la^r$ has the largest size in $ \{\la^r:\  1\leq r \leq \lfloor   t/2\rfloor  \} $ if and only if 
\begin{align*}
\displaystyle
r= \begin{cases}  \lfloor   \alpha_{m,t}(3)\rfloor , \quad &\ \text{when}\  \{  \alpha_{m,t}(3)\}<  1/2; \\ \lfloor   \alpha_{m,t}(3)\rfloor  \ \text{or}\ \lfloor   \alpha_{m,t}(3)\rfloor +1, \quad &\   \text{when}\ \{  \alpha_{m,t}(3)\}=  1/2; \\
  \lfloor   \alpha_{m,t}(3)\rfloor +1, \quad &\  \text{when}\  \{  \alpha_{m,t}(3)\}>  1/2.
\end{cases}
\end{align*}

Also, by \eqref{eq:size_formula} the size of $\mu^s$ is 
\begin{align} \label{eq: mu_s}
|\mu^s|
&= \sum_{j=1}^{s}\left(m(t-2j)+t\binom{m}{2}\right)-\binom{ms}{2} \nonumber
\\&=
-s^2\cdot  (m^2+2m)/{2}+s\cdot  (m^2t+mt-m)/{2}
\\&= 
-  \frac{m^2+2m}{2}\cdot(s-  \alpha_{m,t}(-1))^2+  \frac{m^2+2m}{2}\cdot \alpha^2_{m,t}(-1)
\nonumber.
\end{align}

Case $1$: When $t \equiv 1 \pmod 2$ and $t\leq m+1$, we have 
$$
    \alpha_{m,t}(-1)=  \frac{t-1}{2}+  \frac{m-t+1}{2(m+2)}\geq  \frac{t-1}{2}
.
$$
 Notice that $1\leq s \leq   (t-1)/2$. Then, $\mu^s$  has the largest size in $ \{\mu^s:\  1\leq s \leq \lfloor    (t-1)/2\rfloor  \} $ if and only if 
$$s=  \frac{t-1}{2}=\lfloor   \frac{t-1}{2}\rfloor .$$

Case $2$: When $t \equiv 1 \pmod 2$ and $t> m+1$, we have 
$$
1<  \alpha_{m,t}(-1)<  \frac{t-1}{2}=\lfloor   \frac{t-1}{2}\rfloor 
.
$$
Then, $\mu^s$  has the largest size in $ \{\mu^s:\  1\leq s \leq \lfloor    (t-1)/2\rfloor  \} $  if and only if 
\begin{align*}
s= \begin{cases}  \lfloor   \alpha_{m,t}(-1)\rfloor , \quad &\ \text{when}\  \{  \alpha_{m,t}(-1)\}<  1/2; \\ \lfloor   \alpha_{m,t}(-1)\rfloor  \ \text{or}\ \lfloor   \alpha_{m,t}(-1)\rfloor +1, \quad &\   \text{when}\ \{  \alpha_{m,t}(-1)\}=  1/2; \\
  \lfloor   \alpha_{m,t}(-1)\rfloor +1, \quad &\  \text{when}\  \{  \alpha_{m,t}(-1)\}>  1/2.
\end{cases}
\end{align*}

Case $3$: When $t \equiv 0 \pmod 2$ and $t\leq 2m+3$, we have 
$$
    \alpha_{m,t}(-1)=  \frac{t-2}{2}+  \frac{2m-t+3}{2(m+2)}\geq  \frac{t-2}{2}
.
$$
 Notice that $1\leq s \leq    (t-2)/2$. Then $\mu^s$  has the largest size in $ \{\mu^s:\  1\leq s \leq \lfloor    (t-1)/2\rfloor  \} $ if and only if 
$$s=  \frac{t-2}{2}=\lfloor   \frac{t-1}{2}\rfloor .$$

Case $4$: When $t \equiv 0 \pmod 2$ and $t> 2m+3$, we have 
$$
1<  \alpha_{m,t}(-1)<  \frac{t-2}{2}=\lfloor   \frac{t-1}{2}\rfloor 
.
$$
Then $\mu^s$  has the largest size in $ \{\mu^s:\  1\leq s \leq \lfloor    (t-1)/2\rfloor  \} $  if and only if 
\begin{align*}
s= \begin{cases}  \lfloor   \alpha_{m,t}(-1)\rfloor , \quad &\ \text{when}\  \{  \alpha_{m,t}(-1)\}<  1/2; \\ \lfloor   \alpha_{m,t}(-1)\rfloor  \ \text{or}\ \lfloor   \alpha_{m,t}(-1)\rfloor +1, \quad &\   \text{when}\ \{  \alpha_{m,t}(-1)\}=  1/2; \\
  \lfloor   \alpha_{m,t}(-1)\rfloor +1, \quad &\  \text{when}\  \{  \alpha_{m,t}(-1)\}>  1/2.
\end{cases}
\end{align*}

\textbf{Step 5.} By Step $4$, we obtain the following results for $t\geq 4$ and $m\geq 2$.

Case $1$: When (i) $t \equiv 1 \pmod 2$ and $t> m+1$; or (ii) $t \equiv 0 \pmod 2$ and $t> 2m+3$,  the largest size for partitions in $S^-_{m,t}$ is
\begin{align*}
\begin{cases} \max \left(  F(\lfloor   \alpha_{m,t}(3)\rfloor ), G(\lfloor   \alpha_{m,t}(-1)\rfloor ) \right), \quad & \text{if}\  \{  \alpha_{m,t}(3)\}\leq  1/2 \ \text{and}\  \{  \alpha_{m,t}(-1)\}\leq  1/2; \\ \max \left(  F(\lfloor   \alpha_{m,t}(3)\rfloor ), G(\lfloor   \alpha_{m,t}(-1)\rfloor +1) \right), \quad & \text{if}\  \{  \alpha_{m,t}(3)\}\leq  1/2 \ \text{and}\  \{  \alpha_{m,t}(-1)\}>  1/2; \\
\max \left(  F(\lfloor   \alpha_{m,t}(3)\rfloor +1), G(\lfloor   \alpha_{m,t}(-1)\rfloor ) \right), \quad & \text{if}\  \{  \alpha_{m,t}(3)\}>  1/2 \ \text{and}\  \{  \alpha_{m,t}(-1)\}\leq  1/2;
\\
\max \left(  F(\lfloor   \alpha_{m,t}(3)\rfloor +1), G(\lfloor   \alpha_{m,t}(-1)\rfloor +1) \right), \quad & \text{if}\  \{  \alpha_{m,t}(3)\}>  1/2 \ \text{and}\  \{  \alpha_{m,t}(-1)\}>  1/2.
\end{cases}
\end{align*}

Case $2$: When (iii) $t \equiv 1 \pmod 2$ and $t\leq m+1$; or (iv) $t \equiv 0 \pmod 2$ and $t\leq 2m+3$,  the largest size  for partitions in $S^-_{m,t}$ is
\begin{align*}
\begin{cases} \max \left(  F(\lfloor   \alpha_{m,t}(3)\rfloor ), G(\lfloor   \frac{t-1}{2}\rfloor ) \right), \quad &\ \text{if}\  \{  \alpha_{m,t}(3)\}\leq  1/2; \\ \max \left(  F(\lfloor   \alpha_{m,t}(3)\rfloor +1), G(\lfloor   \frac{t-1}{2}\rfloor ) \right), \quad &\ \text{if}\  \{  \alpha_{m,t}(3)\}>  1/2.
\end{cases}
\end{align*}

It is easy to check that $t=2$ and $3$ can be included in the above Case $2$. Thus this step completes all the cases for the largest size of $(t, mt - 1)$-core partitions with distinct parts. 

\textbf{Step 6.} In this step we claim that the number of partitions in $S^-_{m,t}$ with the largest size is at most~$2$. Otherwise, assume that there exist at least $3$ partitions 
in $S^-_{m,t}$ with the largest size. There are two cases to be considered.

Case $1$: There exist some $1\leq r\leq \lfloor   t/2\rfloor -1$ and $1\leq s\leq \lfloor   (t-1)/2\rfloor $ such that $|\la^r|=|\la^{r+1}|=|\mu^s|$ equals the largest size in   $S^-_{m,t}$. Then $$   \alpha_{m,t}(3)= r+  \frac12 $$ since $|\la^r|=|\la^{r+1}|$.   Therefore
$$ 
r\leq   \alpha_{m,t}(-1)=   \alpha_{m,t}(3)-  \frac{2}{m+2}= r+  \frac12 -  \frac{2}{m+2}<r+  \frac12\leq \lfloor   \frac{t}{2}\rfloor -  \frac12.
$$
Then $s=r$ since $|\mu^s|$ is equal to the largest size in $S^-_{m,t}$. But by \eqref{eq: la_r} and \eqref{eq: mu_s}, $|\la^r|=|\mu^r|$ implies $t=2r$, which contradicts the inequality $r\leq \lfloor   t/2\rfloor -1$. 

Case $2$: There exist some $1\leq s\leq \lfloor   (t-1)/2\rfloor -1$ and $1\leq r\leq \lfloor   t/2\rfloor $ such that $|\mu^s|=|\mu^{s+1}|=|\la^r|$ equals the largest size in   $S^-_{m,t}$. Then $$   \alpha_{m,t}(-1)= s+  \frac12 $$ since $|\mu^s|=|\mu^{s+1}|$.   Therefore
$$ 
s+  \frac12<   \alpha_{m,t}(3)=   \alpha_{m,t}(-1)+  \frac{2}{m+2}= s+  \frac12 +  \frac{2}{m+2}\leq s+1 \leq \lfloor   (t-1)/2\rfloor .
$$
Then $r=s+1$ since $|\mu^s|$ is equal to the largest size in $S^-_{m,t}$. But by \eqref{eq: la_r} and \eqref{eq: mu_s}, $|\la^{s+1}|=|\mu^{s+1}|$ implies $t=2(s+1)$, which contradicts the inequality  $s\leq \lfloor   (t-1)/2\rfloor -1$. 

Then the claim holds by the contradictions in Cases $1$ and $2$. 

\textbf{Step 7.} 
Finally the proof is complete by Steps $5$ and $6$.  
\end{proof}

\section*{Acknowledgments}
The author is supported by Grant P2ZHP2\_171879 of the Swiss National Science Foundation.

\end{document}